\documentclass[11pt,reqno]{amsart}

\usepackage[T1]{fontenc}
\usepackage{graphicx}
\usepackage{amssymb,amsthm,amsmath,mathrsfs,bm,braket,marginnote}
\usepackage{enumerate}
\usepackage{appendix}
\usepackage[colorlinks=true, pdfstartview=FitV, linkcolor=blue, citecolor=blue, urlcolor=blue]{hyperref}
\usepackage{multirow}
\usepackage{subcaption}
\usepackage[linesnumbered,ruled,vlined]{algorithm2e}

\usepackage{graphicx}
\usepackage{booktabs}
\usepackage[english]{babel}

\usepackage{pgf}
\usepackage{pgfplots}
\usepackage{tikz}
\usetikzlibrary{arrows,calc}
\usepackage{verbatim}
\usetikzlibrary{decorations.pathreplacing,decorations.pathmorphing}
\usepackage[numbers,sort&compress]{natbib}
\usepackage{dsfont}

\numberwithin{equation}{section}
\newtheorem{theorem}{Theorem}[section]

\newtheorem{definition}[theorem]{Definition}
\newtheorem{remark}[theorem]{Remark}

\newtheorem{example}[theorem]{Example}

\newtheorem{proposition}[theorem]{Proposition}

 \reversemarginpar

\def\({\left(}
\def\){\right)}
\begin{document}
\title[non-commutative hungry Toda lattice and matrix computation]{Discrete non-commutative hungry Toda lattice and its application in matrix computation}	
\author{Zheng Wang}
\address{School of Mathematical Sciences, Ocean University of China, Qingdao, Shandong, China}
\email{wangzheng9306@stu.ouc.edu.cn}
\author{Shi-Hao Li}
\address{Department of Mathematics, Sichuan University, Chengdu, Sichuan, China}
\email{shihao.li@scu.edu.cn}
\author{Kang-Ya Lu}
\address{School of Applied Science, Beijing Information Science and Technology University, Beijing, China}
\email{lukangya@bistu.edu.cn}
\author{Jian-Qing Sun}
\address{School of Mathematical Sciences, Ocean University of China, Qingdao, Shandong, China}
\email{sunjianqing@ouc.edu.cn}

\date{}

\dedicatory{}


\maketitle

\begin{abstract}
In this paper, we plan to show an eigenvalue algorithm for block Hessenberg matrices by using the idea of non-commutative integrable systems and matrix-valued orthogonal polynomials. 
We introduce adjacent families of matrix-valued $\theta$-deformed bi-orthogonal polynomials, and derive corresponding discrete non-commutative hungry Toda lattice from discrete spectral transformations for polynomials. It is shown that this discrete system can be used as a pre-precessing algorithm for block Hessenberg matrices. Besides, some convergence analysis and numerical examples of this algorithm are presented.
\end{abstract}

\section{Introduction}\label{sec.1}

A resurgence of interest in integrable algorithm came with the equivalence between discrete Toda equation and QR algorithm found by Symes in 1980s \cite{symes82}. Since then, people realized that many famous algorithms have close connections with completely integrable dynamic systems. In \cite{deift83}, it was found by Deift, Nanda and Tomei that the continuous Toda equation could be used to compute eigenvalues for symmetric matrix. Subsequently, Watkins realized that many other decompositions, such as LU decomposition, QR decomposition and Cholesky decomposition, could be fulfilled by isospectral flows \cite{watkins84}. In fact, with the notions of Lie group and Lie algebra introduced into the study of isospectral flows, in \cite{watkins88}, self-similar flows were proposed by considering more general matrix factorizations. Besides, Chu and Norris proposed abstract matrix factorizations generated by subspace decompositions which are not necessarily Lie algebra decompositions in \cite{chu88}. In \cite{deift89}, the authors showed that the above-mentioned matrix factorizations are completely integrable Hamiltonian flows. The development in this direction promotes the interdisciplinary studies of orthogonal polynomials, integrable systems, random matrix theory and numerical analysis, see \cite{deift03,chu08,adler95} and references therein.

In 1990s, it was found by Papageorgiou, Grammaticos and Ramani that the well-known $\varepsilon$-algorithm in convergence acceleration algorithm could be regarded as a discrete potential KdV lattice equation \cite{papageorgiou93}. Subsequently, Nagai, Tokihiro and Satsuma in \cite{nagai1998} demonstrated that the Toda molecule solution is related to the $\varepsilon$-algorithm and Pad\'e approximation. Besides, some of these authors demonstrated that the $\eta$-algorithm is nothing but the disrete KdV equation, and the generalization of $\rho$-algorithm is the cylindrical KdV equation \cite{nagai97}.
 Although it seems that the integrability sheds light on novel interpretations of different algorithms, how to design an integrable algorithm was mysterious at that time. In \cite{nakamura98,nakamura982}, Nakamura and his collaborators used Toda molecule equation to calculate Laplace transforms and a decoding algorithm for BCH Goppa code. Several eigenvalue and singular value algorithms were proposed by considering discrete Lotka-Volterra equation and its hungry analogy \cite{iwasaki04,fukuda09}.
Later, Sun et al. designed an eigenvalue algorithm for structured matrices by using a Bogoyavlensky lattice \cite{sun10}. Besides, there have been many developments in the design of convergence acceleration algorithm by using discrete integrable systems, such as \cite{brezinski12,brezinski11,he11,sun13}.

In recent years, the theory of discrete integrable systems has undergone a true revolution and one of the most fascinating fields is the non-commutative discrete integrable system.
In \cite{gekhtman93}, the inverse spectral problem for non-abelian Toda lattice was considered, and thus matrix-valued orthogonal polynomials have some foundations with non-commutative integrable systems.
The rise of quasi-determinants provided us a powerful algebraic tool to deal with the non-commutativity \cite{gelfand05}. Gelfand et al showed that matrix-valued orthogonal polynomials could be written in terms of quasi-determinants and solutions of some non-commutative integrable systems could be expressed by quasi-determinant as well. It was shown in \cite{li21} that
non-abelian Toda lattice could be derived by the compatibility of the spectral problem and time evolutions of semi-classical matrix-valued orthogonal polynomials. As mentioned, the earlier development of Lie algebra and Lie group decompositions have tight connections with orthogonal polynomials theory and Toda theory.
In non-commutative case, one of the authors in this paper generalized different matrix-valued orthogonal polynomials and derived some new integrable systems \cite{li231,gilson23,wang23}.
It is also worth mentioning that the relationship between non-commutative integrable systems and numerical algorithms has been attracting attention and interest. For example, the non-commutative version of the qd-algorithm derived by Wynn \cite{wynn1963nc} is exactly the full-discrete non-commutative Toda lattices. Brezinski showed how the cross rules of non-scalar $\varepsilon$-algorithms lead to coupled non-Abelian lattice equations in \cite{brezinski2010}. The non-commutative version of the Hermite-Pad\'e type I approximation problem is introduced in \cite{doliwa2022hermite}, and linked with the non-commutative Hirota (discrete Kadomtsev-Petviashvili) system. However, as far as we know, there are currently no examples of designing numerical algorithms for matrix eigenvalues based on non-commutative integrable systems. 

The aim of this paper is to design a new algorithm for block Hessenberg
matrix eigenvalues from a non-commutative integrable system, with the help of quasi-determinants and deformed matrix-valued orthogonal polynomials.
In Section \ref{sec.2}, we give a brief review of matrix-valued $\theta$-deformed bi-orthogonal polynomials. It was shown that the spectral operators for matrix-valued $\theta$-deformed bi-orthogonal polynomials could be written as upper/lower Hessenberg matrices with block matrix elements, and these banded block Hessenberg matrices are the objects we design algorithms for. As algorithm should correspond to some discrete dynamic systems, in Section \ref{sec.3}, we introduce an adjacent families of matrix-valued $\theta$-deformed bi-orthogonal polynomials and study their discrete spectral transformations. From the compatibility conditions of these discrete spectral transformations, the non-commutative hungry Toda lattices is derived.
We further show that discrete spectral transformations are equivalent to the recurrence relations, and thus the iteration process for eigenvalues could be realized by discrete spectral transformations.
In Section \ref{sec.4}, the asymptotic convergence of the non-commutative hungry Toda lattices is given, based on which we obtain a pre-processing algorithm for block Hessenberg matrix eigenvalues. Some numerical examples are presented in Section \ref{sec.5}. Section \ref{sec.6} is devoted to conclusions and discussions.

\section{A brief review of matrix-valued $\theta$-deformed bi-orthogonal polynomials}\label{sec.2}
In this section, we give a brief review to matrix-valued $\theta$-deformed bi-orthogonal polynomials, whose recurrence relation give rise to some upper and lower Hessenberg matrices. As we intend to propose some block algorithms for Hessenberg matrices, let's start with an introduction to $\theta$-deformed bilinear form.

\subsection{A $\theta$-deformed bilinear form}
In this part, we introduce a $\theta$-deformed bilinear form, which is a generalization of inner product for standard orthogonal polynomials. Such an idea could be dated back to earlier results of Konhauser \cite{konhauser67} and Carlitz \cite{carlitz68}, in which the $\theta$-deformed bilinear form is defined by $\langle \cdot,\cdot\rangle_{\theta}:\, \mathbb{R}[x]\times \mathbb{R}[x]\to\mathbb{R}$ such that
\begin{align}\label{tbilinear}
\langle f(x),g(x)\rangle_\theta=\int_{\mathbb{R}}f(x)g(x^\theta)d\mu(x),\quad f(x),\,g(x)\in\mathbb{R}[x],
\end{align}
where $d\mu(x)$ is a positive measure compactly supported on intervals of $\mathbb{R}$.

To make this bilinear form defined on the space of polynomials with matrix-valued coefficients, we consider a matrix-valued Radon measure $\mu:(-\infty,\infty)\rightarrow \mathbb{R}^{p\times p}$. Without any abuse of notations, we use $\mathbb{R}^{p\times p}$ to stand for the set of all $p\times p$  matrices with real elements. Moreover, if we normalize the Radon measure $\mu(\mathbb{R})=\mathbb{I}_{p}$, where $\mathbb{I}_{p}$ is a $p\times p$ identity matrix, then according to the Radon-Nikodym theorem, the measure $\mu(\mathbb{R})$ is related to a matrix-valued weight function $W(x)$, such that $d\mu(x)=W(x)dx$. Details could be referred to Damanik et al \cite{damanik07} where the analytic theory of matrix-valued polynomials is given.
Therefore, with a matrix-valued weight function $W(x)$, we generalize the $\theta$-deformed bilinear form \eqref{tbilinear} to a matrix-valued $\theta$-deformed bilinear form
\begin{align}\label{b}
\langle\cdot,\cdot\rangle_{\theta}:\mathbb{R}^{p\times p}[x]\times\mathbb{R}^{p\times p}[x]\rightarrow\mathbb{R}^{p\times p},\quad \langle f(x),g(x)\rangle_{\theta}=\int_{\mathbb{R}}f(x)W(x)g^{\top}(x^{\theta})dx,
\end{align}
where $\top$ means the transpose of a matrix.
We should remark that the parameter $\theta$ could be taken as any value in $\mathbb{R}_+$. However, in this paper, we focus mainly on $\theta\in \mathbb{Z}_+$ case, in which an integrable eigenvalue algorithm for block upper/lower Hessenberg matrices could be given.

Importantly, such a bilinear form has some basic property from its definition, as stated in \cite[proposition 2.1]{gilson23}.
\begin{proposition}
The bilinear form \eqref{b} has the following properties:
\begin{enumerate}
\item {\textbf{Bimodule structures}}. For any $L_1,\,L_2,\,R_1,\,R_2\in\mathbb{R}^{p\times p}$ and $f_1(x),f_2(x),g_1(x),g_2(x)\in\mathbb{R}^{p\times p}[x]$, we have
 \begin{align}\label{bimodule}
 \begin{aligned}
& {\langle L_1f_1(x)+L_2f_2(x),g(x)\rangle}_\theta=L_1{\langle f_1(x), g(x)\rangle}_\theta+L_2{\langle f_2(x),g(x)\rangle}_\theta,\\
& {\langle f(x),R_1 g_1(x)+R_2 g_2(x)\rangle}_\theta={\langle f(x),g_1(x)\rangle}_\theta R_1^\top+{\langle f(x),g_2(x)\rangle}_\theta R_2^\top.
\end{aligned}
\end{align}
\item {\textbf{Quasi-symmetry property}.} For any $f(x),g(x)\in\mathbb{R}^{p\times p}[x]$, it holds that
\begin{align}\label{qs}
&\langle x^\theta f(x), g(x)\rangle_\theta = \langle f(x),x g(x)\rangle_\theta,\quad \text{~for~} \theta\in\mathbb{Z}_+.
\end{align}
\end{enumerate}
\end{proposition}

\subsection{Matrix-valued $\theta$-deformed bi-orthogonal polynomials}
Now we could define a bi-orthogonal basis in $\mathbb{R}^{p \times p}[x]$ in terms of the bilinear form \eqref{b}.
\begin{definition}
Let $\{P_n(x)\}_{n\in\mathbb{N}}$ and $\{Q_n(x)\}_{n\in\mathbb{N}}$ be two sequences of matrix-valued orthogonal polynomials, whose coefficients are in the ring $\mathbb{R}^{p \times p}$.
We call $\{P_n(x)\}_{n\in\mathbb{N}}$ and $\{Q_n(x)\}_{n\in\mathbb{N}}$ are bi-orthogonal with each other under the bilinear form \eqref{b}, if they satisfy
\begin{align*}
\langle P_n(x),Q_m(x)\rangle_\theta=H_n\delta_{n,m},
\end{align*}
where deg $P_n=n$ and deg $Q_m=m$, and $H_n$ is a proper normalization factor.
\end{definition}
The existence of these bi-orthogonal polynomials depends on the moment sequence
\begin{align}\label{moments}
m_{i+j\theta}=\langle x^i\mathbb{I}_p,x^j\mathbb{I}_p\rangle_\theta=\int_{\mathbb{R}}x^{i+j\theta}W(x)dx,\qquad i,j=0,1,\cdots.
\end{align}
If we require conditions
\begin{enumerate}
\item all moments $\{m_{i+j\theta}\}_{i,j=0,1,\cdots}$ exist and are finite,
\item corresponding moment matrices $\left(
m_{i+j\theta}
\right)_{i,j=0,1,\cdots}$ are invertible,
\end{enumerate} then matrix-valued $\theta$-deformed bi-orthogonal polynomials $\{P_n(x)\}_{n\in\mathbb{N}}$ and $\{Q_n(x)\}_{n\in\mathbb{N}}$ exist and uniquely determined. This is called the moment condition for this $\theta$-deformed bilinear form.

Obviously, the bi-orthogonality doesn't break if we consider an invertible transformation
\begin{align*}
P_n(x)\to A_n \tilde{P}_n(x),\quad Q_n(x)\to B_n\tilde{Q}_n(x), \quad A_n,\,B_n\in \text{GL}_p(\mathbb{R}).
\end{align*}
Therefore, we could take representative of these equivalence classes, and assume that $\{P_n(x)\}_{n\in\mathbb{N}}$ and $\{Q_n(x)\}_{n\in\mathbb{N}}$ should be monic, i.e. the coefficients of the highest order are $\mathbb{I}_p$.
In this case, $H_n$ is also uniquely determined, and we have the following proposition.
\begin{proposition}
If $\{P_n(x)\}_{n\in\mathbb{N}}$ and $\{Q_n(x)\}_{n\in\mathbb{N}}$ are monic matrix-valued polynomials, biorthogonal with each other under the $\theta$-deformed bilinear form \eqref{b}, then they have the following quasi-determinant expressions\footnote{Some basic notations for quasi-determinant are given in Appendix.
}
\begin{align}\label{pq}
\begin{aligned}
P_{n}(x)=\left|\begin{array}{ccccc}
m_{0}&m_{\theta}&\cdots&m_{(n-1)\theta}&\mathbb{I}_{p}\\
m_{1}&m_{\theta+1}&\cdots&m_{(n-1)\theta+1}&x\mathbb{I}_{p}\\
\vdots&\vdots&&\vdots&\vdots\\
m_{n}&m_{\theta+n}&\cdots&m_{(n-1)\theta+n}&\fbox{$x^{n}\mathbb{I}_{p}$}
\end{array}\right|, \\
{Q^\top_n(x)}=\left|\begin{array}{cccc}
m_{0}&\cdots&m_{(n-1)\theta}&m_{n\theta}\\
\vdots& &\vdots&\vdots\\
m_{n-1}&\cdots&m_{n-1+(n-1)\theta}&m_{n-1+n\theta}\\
\mathbb{I}_p&\cdots&x^{n-1}\mathbb{I}_p&\boxed{x^n\mathbb{I}_p}
\end{array}
\right|,
\end{aligned}
\end{align}
and the normalization factor $H_n$ is given by
\begin{align*}
H_{n}=\left|\begin{array}{ccccc}
m_{0}&m_{\theta}&\cdots&m_{(n-1)\theta}&m_{n\theta}\\
m_{1}&m_{\theta+1}&\cdots&m_{(n-1)\theta+1}&m_{n\theta+1}\\
\vdots&\vdots&&\vdots&\vdots\\
m_{n}&m_{\theta+n}&\cdots&m_{(n-1)\theta+n}&\fbox{$m_{n\theta+n}$}
\end{array}\right|.
\end{align*}
\end{proposition}

\section{Adjacent families, discrete spectral transformations and discrete non-commutative hungry Toda lattice}\label{sec.3}

In this section, we introduce the concept of adjacent families of matrix-valued $\theta$-deformed bi-orthogonal polynomials. Adjacent families of orthogonal polynomials are important in the characterization of discrete spectral transformations \cite{brezinski02}, which are used to design for qd-algorithm due to Rutishauser \cite{rutishauser57} and LR transformation for the computation of the eigenvalues of a matrix.
By using adjacent families, we derive non-commutative discrete integrable systems from spectral transformations. In commutative case, such integrable systems were called hungry extensions of the discrete Toda equations \cite{Nakamura15}.
\subsection{Adjacent families of matrix-valued $\theta$-deformed bi-orthogonal polynomials}

For any $\alpha\in\mathbb{Z}$, we could introduce an adjacent families of $\theta$-deformed bilinear form
\begin{align}\label{72}
\langle f(x),g(x)\rangle_{\theta}^{\alpha}=\int_{\mathbb{R}}f(x)W^{(\alpha)}(x)g^{\top}(x^{\theta})dx,\quad W^{(\alpha)}(x)=x^{\alpha}W(x).
\end{align}
Obviously, this bilinear form inherits the bimodule property \eqref{bimodule} and quasi-symmetry property \eqref{qs}.
Moreover, from the bilinear form (\ref{72}), we can induce a family of monic bi-orthogonal polynomial sequences $\{P_{n}^{(\alpha)}(x), Q_{m}^{(\alpha)}(x)\}_{n,m\in\mathbb{N}} $ such that
\begin{align}\label{21}
\langle P_{n}^{(\alpha)}(x),Q_{m}^{(\alpha)}(x)\rangle_{\theta}^{\alpha}=H_{n}^{(\alpha)}\delta_{n,m},
\end{align}
where $H_{n}^{(\alpha)}\in\text{GL}_p(\mathbb{R})$ is a nonsingular normalization factor.
In fact, the bi-orthogonality in \eqref{21} could be equivalently written as
\begin{align}\label{biorthogonality}
\begin{aligned}
&\langle P_n^{(\alpha)}(x),x^i\mathbb{I}_p\rangle_\theta^\alpha=0,\quad 0\leq i\leq n-1,\\
&\langle P_n^{(\alpha)}(x),x^n\mathbb{I}_p\rangle_\theta^\alpha=H_n^{(\alpha)}.
\end{aligned}
\end{align}
By a direct calculation, we can get quasi-determinant expressions for $P_{n}^{(\alpha)}(x)$ and $Q_{n}^{(\alpha)}(x)$, where
\begin{align}
P_{n}^{(\alpha)}(x)=\left|\begin{array}{ccccc}
m_{\alpha}&m_{\alpha+\theta}&\cdots&m_{\alpha+(n-1)\theta}&\mathbb{I}_{p}\\
m_{\alpha+1}&m_{\alpha+\theta+1}&\cdots&m_{\alpha+(n-1)\theta+1}&x\mathbb{I}_{p}\\
\vdots&\vdots&&\vdots&\vdots\\
m_{\alpha+n}&m_{\alpha+\theta+n}&\cdots&m_{\alpha+(n-1)\theta+n}&\fbox{$x^{n}\mathbb{I}_{p}$}
\end{array}\right|,
\end{align}
and
\begin{align}
(Q_{n}^{(\alpha)}(x))^{\top}=\left|\begin{array}{ccccc}
m_{\alpha}&m_{\alpha+\theta}&\cdots&m_{\alpha+n\theta}\\
m_{\alpha+1}&m_{\alpha+\theta+1}&\cdots&m_{\alpha+n\theta+1}\\
\vdots&\vdots&&\vdots\\
m_{\alpha+n-1}&m_{\alpha+\theta+n-1}&\cdots&m_{\alpha+n\theta+n-1}\\
\mathbb{I}_{p}&x\mathbb{I}_{p}&\cdots&\fbox{$x^{n}\mathbb{I}_{p}$}
\end{array}\right|.
\end{align}
We call $\{P_{n}^{(\alpha)}(x)\}_{n\in\mathbb{N}}$ and $\{Q_{n}^{(\alpha)}(x)\}_{n\in\mathbb{N}}$ the $\alpha$-th adjacent family of $\{P_{n}(x)\}_{n\in\mathbb{N}}$ and $\{Q_{n}(x)\}_{n\in\mathbb{N}}$ respectively.

Furthermore, if we put the quasi-determinant expression of $P_{n}^{(\alpha)}(x)$ and $Q_{n}^{(\alpha)}(x)$ into orthogonal relation (\ref{21}), then the normalization factor $H^{(\alpha)}_{n}$ can be expressed in terms of Hankel quasi-determinant and
\begin{align}
H_{n}^{(\alpha)}=\left|\begin{array}{ccccc}
m_{\alpha}&m_{\alpha+\theta}&\cdots&m_{\alpha+(n-1)\theta}&m_{\alpha+n\theta}\\
m_{\alpha+1}&m_{\alpha+\theta+1}&\cdots&m_{\alpha+(n-1)\theta+1}&m_{\alpha+n\theta+1}\\
\vdots&\vdots&&\vdots&\vdots\\
m_{\alpha+n}&m_{\alpha+\theta+n}&\cdots&m_{\alpha+(n-1)\theta+n}&\fbox{$m_{\alpha+n\theta+n}$}
\end{array}\right|.
\end{align}

A significant property for matrix-valued $\theta$-deformed bi-orthogonal polynomials is the recurrence relation. Regarding with adjacent families, we have the following proposition.
\begin{proposition}
There exists the following recurrence relation for $\{P_n^{(\alpha)}(x)\}_{n\in\mathbb{N},\alpha\in\mathbb{Z}}$ such that
\begin{align}\label{recurrencep}
x^\theta P^{(\alpha)}_n(x)=P^{(\alpha)}_{n+\theta}(x)+\sum_{j=n-1}^{n+\theta-1}\kappa_{n,j}^{(\alpha)} P_j^{(\alpha)}(x),
\end{align}
where
\begin{align*}
\kappa_{n,j}^{(\alpha)}=\langle P_n^{(\alpha)}(x),xQ_{j}^{(\alpha)}(x)\rangle_\theta^{\alpha}\cdot \left(H_j^{(\alpha)}\right)^{-1}.
\end{align*}
Besides, the recurrence relation for $\{Q_n^{(\alpha)}(x)\}_{n\in\mathbb{N},\alpha\in\mathbb{Z}}$ is given by
\begin{align}\label{recurrenceq}
xQ_n^{(\alpha)}(x)=Q_{n+1}^{(\alpha)}(x)+\sum_{j=n-\theta}^n \ell_{n,j}^{(\alpha)}Q_j^{(\alpha)}(x),
\end{align}
where
\begin{align*}
\left(\ell_{n,j}^{(\alpha)}\right)^\top=\langle P_j^{(\alpha)}(x),xQ_n^{(\alpha)}(x)\rangle_\theta^{\alpha}\left(H_j^{(\alpha)}\right)^{-1}.
\end{align*}
\end{proposition}

\subsection{Spectral transformations for $\{P_n^{(\alpha)}(x)\}_{n\in\mathbb{N},\alpha\in\mathbb{Z}}$ and discrete non-commutative hungry Toda-I lattice}
In this part, we first show spectral transformations, including the Christoffel transformation and the Geronimus transformation, for $\{P_n^{(\alpha)}(x)\}_{n\in\mathbb{N},\alpha\in\mathbb{Z}}$. Besides, we demonstrate that the discrete non-commutative hungry Toda-I lattice could be derived from the compatibility condition of these spectral transformations. In the following theorem, we show the Christoffel transformation for adjacent families of $\{P_n^{(\alpha)}(x)\}_{n\in\mathbb{N},\alpha\in\mathbb{Z}}$.

\begin{theorem}\label{3}
The Christoffel transformation of adjacent families of matrix-valued $\theta$-deformed bi-orthogonal polynomials $\{P_{n}^{(\alpha)}(x)\}_{n\in\mathbb{N},\alpha\in\mathbb{Z}}$ is given by
\begin{align}\label{1p}
xP_{n}^{(\alpha+1)}(x)=P_{n+1}^{(\alpha)}(x)+\omega_{n+1}^{(\alpha)}P_{n}^{(\alpha)}(x),
\end{align}
where
$
\omega_{n+1}^{(\alpha)}=H_{n}^{(\alpha+1)}(H_{n}^{(\alpha)})^{-1}.
$
\end{theorem}
\begin{proof}
In the quasi-determinant expressions of $P_{n+1}^{(\alpha)}(x)$, if we apply the non-commutative version of Jacobi identity (\ref{51}) to the first and $(n+2)$-th rows and the (n+1)-th and $(n+2)$-th columns, then we get
\begin{align*}
&\left|\begin{array}{ccccc}
m_{\alpha}&m_{\alpha+\theta}&\cdots&m_{\alpha+n\theta}&\mathbb{I}_{p}\\
m_{\alpha+1}&m_{\alpha+\theta+1}&\cdots&m_{\alpha+n\theta+1}&x\mathbb{I}_{p}\\
\vdots&\vdots&&\vdots&\vdots\\
m_{\alpha+n+1}&m_{\alpha+\theta+n+1}&\cdots&m_{\alpha+n\theta+n+1}&\fbox{$x^{n+1}\mathbb{I}_{p}$}
\end{array}\right|
=\left|\begin{array}{ccccc}
m_{\alpha+1}&\cdots&m_{\alpha+(n-1)\theta+1}&x\mathbb{I}_{p}\\
m_{\alpha+2}&\cdots&m_{\alpha+(n-1)\theta+2}&x^{2}\mathbb{I}_{p}\\
\vdots&&\vdots&\vdots\\
m_{\alpha+n+1}&\cdots&m_{\alpha+(n-1)\theta+n+1}&\fbox{$x^{n+1}\mathbb{I}_{p}$}
\end{array}\right|
\end{align*}
\begin{align*}
&-\left|\begin{array}{ccccc}
m_{\alpha+1}&\cdots&m_{\alpha+n\theta+1}\\
\vdots&&\vdots\\
m_{\alpha+n+1}&\cdots&\fbox{$m_{\alpha+n\theta+n+1}$}
\end{array}\right|\left|\begin{array}{ccccc}
m_{\alpha}&\cdots&\fbox{$m_{\alpha+n\theta}$}\\
\vdots&&\vdots\\
m_{\alpha+n}&\cdots&m_{\alpha+n\theta+n}
\end{array}\right|^{-1}\left|\begin{array}{ccccc}
m_{\alpha}&\cdots&m_{\alpha+(n-1)\theta}&\fbox{$\mathbb{I}_{p}$}\\
\vdots&&\vdots\\
m_{\alpha+n}&\cdots&m_{\alpha+(n-1)\theta+n}&x^{n}\mathbb{I}_{p}
\end{array}\right|.
\end{align*}
By making use of homological relations (\ref{qq}) to the last two quasi-determinants, we get
\begin{align}
P_{n+1}^{(\alpha)}(x)=xP_{n}^{(\alpha+1)}(x)-H_{n}^{(\alpha+1)}(H_{n}^{(\alpha)})^{-1}P_{n}^{(\alpha)}(x).
\end{align}
So far, we complete the proof of this theorem.
\end{proof}
Besides, a Geronimus transformation for adjacent families could be obtained by simply using the bi-orthogonal relation.
\begin{theorem}
The Geronimus transformation of adjacent families of matrix-valued $\theta$-deformed bi-orthogonal polynomials $\{P_{n}^{(\alpha)}(x)\}_{n\in\mathbb{N}}$ is given by
\begin{align}\label{2p}
P_{n}^{(\alpha)}(x)=P_{n}^{(\alpha+\theta)}(x)+\varepsilon_{n}^{(\alpha)}P_{n-1}^{(\alpha+\theta)}(x),
\end{align}
where
$
\varepsilon_{n}^{(\alpha)}=H_{n}^{(\alpha)}(H_{n-1}^{(\alpha+\theta)})^{-1}.
$
\end{theorem}
\begin{proof}
As $\{P_n^{(\alpha+\theta)}\}_{n\in\mathbb{N}}$ form a basis in $\mathbb{R}^{p\times p}[x]$, we could expand $P_n^{(\alpha)}(x)$ in terms of $\{P_n^{(\alpha+\theta)}\}_{n\in\mathbb{N}}$, which indicates that
\begin{align}\label{bas}
P_{n}^{(\alpha)}(x)=P_{n}^{(\alpha+\theta)}(x)+\beta_{n,n-1}^{(\alpha)}P_{n-1}^{(\alpha+\theta)}(x)+\cdots+\beta_{n,0}^{(\alpha)}P_{0}^{(\alpha+\theta)}(x).
\end{align}
By acting linear functionals
\begin{align*}
\mathcal{L}_i[\cdot]:=\int_{\mathbb{R}}\cdot W^{(\alpha+\theta)}(x)Q_i^{(\alpha+\theta)}(x)dx,\quad  i=0,1,\cdots,n
\end{align*}
on both sides of \eqref{bas}, and by the bi-orthogonality \eqref{biorthogonality}, one finds that
\[\nonumber
\langle P_{n}^{(\alpha)}(x),Q_{i}^{(\alpha+\theta)}(x)\rangle_{\theta}^{\alpha+\theta}=\langle P_{n}^{(\alpha)}(x),xQ_{i}^{(\alpha+\theta)}(x)\rangle_{\theta}^{\alpha}=0
\]
if $i<n-1$. Moreover, when $i=n-1$, it results in
\[\nonumber
\beta_{n,n-1}^{(\alpha)}=\langle P_{n}^{(\alpha)}(x),xQ_{n-1}^{(\alpha+\theta)}(x)\rangle_{\theta}^{\alpha}\cdot(H_{n-1}^{(\alpha+\theta)})^{-1}=H_{n}^{(\alpha)}(H_{n-1}^{(\alpha+\theta)})^{-1}.
\]
\end{proof}
It should be remarked that Christoffel and Geronimus transformations could also be written as a matrix form. Let's denote
\begin{align*}
\Phi^{(\alpha)}=\left(\begin{array}{c}
P_0^{(\alpha)}(x)\\P_1^{(\alpha)}(x)\\\vdots
\end{array}
\right),
\end{align*}
then the Christoffel transformation (\ref{1p}) could be written as
\begin{align}\label{i}
x\Phi^{(\alpha+1)}=\mathscr{A}^{(\alpha)}\Phi^{(\alpha)},\quad \mathscr{A}^{(\alpha)}=\left(\begin{array}{cccc}
\omega_{1}^{(\alpha)}&\mathbb{I}_{p}&&\\
&\omega_{2}^{(\alpha)}&\mathbb{I}_{p}&\\
&&\ddots&\ddots
\end{array}\right),
\end{align}
and the Geronimus transformation (\ref{2p}) could be written as
\begin{align}\label{j}
\Phi^{(\alpha)}=\mathscr{B}^{(\alpha)}\Phi^{(\alpha+\theta)},\quad  \mathscr{B}^{(\alpha)}=\left(\begin{array}{cccc}
\mathbb{I}_{p}&&&\\
\varepsilon_{1}^{(\alpha)}&\mathbb{I}_{p}&&\\
&\varepsilon_{2}^{(\alpha)}&\mathbb{I}_{p}&\\
&&\ddots&\ddots
\end{array}\right).
\end{align}
Therefore, the compatibility condition of (\ref{i}) and (\ref{j}) can result in
\begin{align}
\mathscr{A}^{(\alpha)}\mathscr{B}^{(\alpha)}=\mathscr{B}^{(\alpha+1)}\mathscr{A}^{(\alpha+\theta)}.
\end{align}
This is equivalent to the nonlinear equation
\begin{align}\label{3q}
\begin{aligned}
\omega_{n+1}^{(\alpha)}&\varepsilon_{n}^{(\alpha)}=\varepsilon_{n}^{(\alpha+1)}\omega_{n}^{(\alpha+\theta)},\\
\omega_{n}^{(\alpha+\theta)}&+\varepsilon_{n-1}^{(\alpha+1)}=\omega_{n}^{(\alpha)}+\varepsilon_{n}^{(\alpha)},
\end{aligned}
\end{align}
which is a fully discrete non-commutative hungry Toda-I lattice.
\begin{remark}
The equation \eqref{3q} could be expressed in terms of ``non-commutative $\tau$-function'' $\{H_n^{(\alpha)}\}_{n\in\mathbb{N},\alpha\in\mathbb{Z}}$, and we have
\begin{align}\label{be}
H_{n+1}^{(\alpha)}=H_{n}^{(\alpha+\theta+1)}+H_{n}^{(\alpha+\theta)}\((H_{n-1}^{\alpha+\theta+1})^{-1}-(H_{n}^{\alpha})^{-1}\)H_{n}^{(\alpha+1)},
\end{align}
\end{remark}

\subsection{Spectral transformations for $\{Q_n^{(\alpha)}(x)\}_{n\in\mathbb{N},\alpha\in\mathbb{Z}}$ and discrete non-commutative hungry Toda-II lattice}
In this part, we pay our attention to spectral transformations for adjacent families $\{Q_n^{(\alpha)}(x)\}_{n\in\mathbb{N},\alpha\in\mathbb{Z}}$, and obtain the discrete non-commutative hungry Toda-II lattice from compatibility conditions. Since the proof process is similar to those of $\{P_n^{(\alpha)}(x)\}_{n\in\mathbb{N},\alpha\in\mathbb{Z}}$, we omit the details here.
\begin{theorem}
The Christoffel transformation of adjacent families for matrix-valued $\theta$-deformed bi-orthogonal polynomials $\{Q_{n}^{(\alpha)}(x)\}_{n\in\mathbb{N},\alpha\in\mathbb{Z}}$ is given by
\begin{align}
xQ_{n}^{(\alpha+\theta)}(x)=Q_{n+1}^{(\alpha)}(x)+q_{n+1}^{(\alpha)}Q_{n}^{(\alpha)}(x),\label{h}
\end{align}
where
$(q_{n+1}^{(\alpha)})^{\top}=(H_{n}^{(\alpha)})^{-1}H_{n}^{(\alpha+\theta)}.$
\end{theorem}
\begin{theorem}
The Geronimus transformation for adjacent families of matrix-valued $\theta$-deformed bi-orthogonal polynomials $\{Q_{n}^{(\alpha)}(x)\}_{n\in\mathbb{N},\alpha\in\mathbb{Z}}$ is given by
\begin{align}
Q_{n}^{(\alpha)}(x)=Q_{n}^{(\alpha+1)}(x)+e_{n}^{(\alpha)}Q_{n-1}^{(\alpha+1)}(x),\label{2q}
\end{align}
where
$(e_{n}^{(\alpha)})^{\top}=(H_{n-1}^{(\alpha+1)})^{-1}H_{n}^{(\alpha)}.$
\end{theorem}
In this case, if we denote
\begin{align*}
\Psi^{(\alpha)}=\left(\begin{array}{c}
Q_0^{(\alpha)}(x)\\
Q_1^{(\alpha)}(x)\\
\vdots
\end{array}
\right),
\end{align*}
then its corresponding Christoffel transformation (\ref{h}) could be written as
\begin{align}\label{k}
x\Psi^{(\alpha+\theta)}=\mathscr{R}^{(\alpha)}\Psi^{(\alpha)},\quad \mathscr{R}^{(\alpha)}=\left(\begin{array}{cccc}
q_{1}^{(\alpha)}&\mathbb{I}_{p}&&\\
&q_{2}^{(\alpha)}&\mathbb{I}_{p}&\\
&&\ddots&\ddots
\end{array}\right),
\end{align}
and the Geronimus transformation (\ref{2q}) could be written as
\begin{align}\label{l}
\Psi^{(\alpha)}=\mathscr{L}^{(\alpha)}\Psi^{(\alpha+1)},\quad  \mathscr{L}^{(\alpha)}=\left(\begin{array}{cccc}
\mathbb{I}_{p}&&&\\
e_{1}^{(\alpha)}&\mathbb{I}_{p}&&\\
&e_{2}^{(\alpha)}&\mathbb{I}_{p}&\\
&&\ddots&\ddots
\end{array}\right).
\end{align}
Moreover, the compatibility condition of (\ref{k}) and (\ref{l}) results in
\begin{align}\label{identity}
\mathscr{R}^{(\alpha)}\mathscr{L}^{(\alpha)}=\mathscr{L}^{(\alpha+\theta)}\mathscr{R}^{(\alpha+1)},
\end{align}
which could be written as nonlinear equations
\begin{align} \label{3p}
\begin{split}
q_{n+1}^{(\alpha)}&e_{n}^{(\alpha)}=e_{n}^{(\alpha+\theta)}q_{n}^{(\alpha+1)},\\
q_{n}^{(\alpha+1)}&+e_{n-1}^{(\alpha+\theta)}=q_{n}^{(\alpha)}+e_{n}^{(\alpha)}.
\end{split}
\end{align}
The equation (\ref{3p}) is called the fully discrete non-commutative hungry Toda-II lattice.
\begin{remark}
Although equations \eqref{3q} and \eqref{3p} have different nonlinear forms, they admit the same ``bilinear'' form which is written by $\{H_n^{(\alpha)}\}_{n\in\mathbb{N},\alpha\in\mathbb{Z}}$. In other words, if we substitute $(q_{n+1}^{(\alpha)})^\top=(H_n^{(\alpha)})^{-1}H_n^{(\alpha+\theta)}$ and $(e_{n}^{(\alpha)})^{\top}=(H_{n-1}^{(\alpha+1)})^{-1}H_{n}^{(\alpha)}$ into \eqref{3p}, then we could  get \eqref{be} again.
\end{remark}
\section{Generalized block qd-algorithm deduced by non-commutative hungry Toda lattice}\label{sec.4}
In this section, we consider the finite truncation for the non-commutative discrete hungry Toda lattice, from which a pre-processing algorithm for eigenvalues are proposed. We refer to this algorithm as the generalized block qd algorithm, since it could be reduced to the block qd-algorithm in a special case. We also conduct an analysis of the convergence of the generalized block qd algorithm.
\subsection{Finite truncation for non-commutative discrete hungry Toda system}\label{sec.4.1}
To explore potential numerical applications, we first impose some zero boundary conditions on the  non-commutative discrete hungry Toda-II lattice, resulting in the following finite-dimensional system,
\begin{align}\label{toda2}
\left\{\begin{array}{llll}
q_{m}^{(\alpha+1)}+e_{m-1}^{(\alpha+\theta)}=q_{m}^{(\alpha)}+e_{m}^{(\alpha)},\\
m=1,2,\cdots,n,\quad \alpha=0,1,\cdots,\\
e_{m}^{(\alpha+\theta)}q_{m}^{(\alpha+1)}=q_{m+1}^{(\alpha)}e_{m}^{(\alpha)},\\
m=1,2,\cdots,n-1,\quad \alpha=0,1,\cdots\\
e_{0}^{(\alpha)}\equiv\mathbb{O}_{p},\quad e_{n}^{(\alpha)}\equiv\mathbb{O}_{p},\quad \alpha=0,1,\cdots,
\end{array}\right.
\end{align}
where $\mathbb{O}_{p}$ is a $p\times p$ null matrix, and $n\in\mathbb{Z}_+$.
Under such circumstance, if we denote
\begin{align*}
\Psi_n^{(\alpha)}=\left(\begin{array}{c}
Q_0^{(\alpha)}(x)\\
Q_1^{(\alpha)}(x)\\
\vdots\\
Q_{n-1}^{(\alpha)}(x)
\end{array}
\right),
\end{align*}
the recurrence relation \eqref{recurrenceq} could be written as
\begin{align}\label{pt}
x\Psi_n^{(\alpha)}=\mathscr{J}_n^{(\alpha)}\Psi_n^{(\alpha)},\quad \mathscr{J}_n^{(\alpha)}=\Lambda+\ell^{(\alpha)}_0\Lambda^0+\ell^{(\alpha)}_1\Lambda^\top+\cdots+\ell^{(\alpha)}_\theta (\Lambda^\top)^\theta,
\end{align}
where
\begin{align*}
\Lambda=\left(\begin{array}{cccccc}
0&\mathbb{I}_p&0&0&\cdots&0\\
0&0&\mathbb{I}_p&0&\cdots&0\\
0&0&0&\mathbb{I}_p&\cdots&0\\
\vdots&\vdots&\vdots&\vdots&\ddots&\vdots\\
0&0&0&0&\cdots&\mathbb{I}_p\\
0&0&0&0&\cdots&0
\end{array}
\right),
\end{align*}
and $\ell_i^{(\alpha)}=\text{diag}(\ell_{i,0}^{(\alpha)},\ell_{i+1,1}^{(\alpha)},\cdots,\ell_{i+n-1,n-1}^{(\alpha)})$.
Moreover, block matrices $\mathscr{R}^{(\alpha)}$ and $\mathscr{L}^{(\alpha)}$ in \eqref{k} and \eqref{l} can be truncated, and we have
\begin{align}\label{dpt}
x\Psi_n^{(\alpha+\theta)}=\mathscr{R}_n^{(\alpha)}\Psi_n^{(\alpha)},\quad \Psi_n^{(\alpha)}=\mathscr{L}_n^{(\alpha)}\Psi_n^{(\alpha+1)},
\end{align}
with operators
\[
\mathscr{R}_{n}^{(\alpha)}=\Lambda+q^{(\alpha)}\Lambda^0,\quad
\mathscr{L}_n^{(\alpha)}=\Lambda^0+e^{(\alpha)}\Lambda^\top,
\]
where $q^{(\alpha)}=\text{diag}(q_1^{(\alpha)},\cdots,q_n^{(\alpha)})$ and $e^{(\alpha)}=\text{diag}(e_0^{(\alpha)},\cdots,e_{n-1}^{(\alpha)})$.

We show in the next theorem about relations between spectral operators.
\begin{theorem}
For any $\alpha\in\mathbb{N}$, spectral operators $\mathscr{J}_n^{(\alpha)}$, $\mathscr{L}_n^{(\alpha)}$ and $\mathscr{R}_n^{(\alpha)}$ have following relations
\begin{subequations}
\begin{align}
&\mathscr{R}_n^{(\alpha)}\mathscr{L}_n^{(\alpha)}=\mathscr{L}_{n}^{(\alpha+\theta)}\mathscr{R}_n^{(\alpha+1)},\label{sub1}\\
&\mathscr{J}_n^{(\alpha)}=\mathscr{L}_n^{(\alpha)}\mathscr{L}_n^{(\alpha+1)}\cdots\mathscr{L}_n^{(\alpha+\theta-1)}\mathscr{R}_n^{(\alpha)},\label{sub2}\\
&\mathscr{J}_{n}^{(\alpha+1)}=\left(
\mathscr{L}_n^{(\alpha)}
\right)^{-1}\mathscr{J}_n^{(\alpha)}\mathscr{L}_n^{(\alpha)}.\label{sub3}
\end{align}
\end{subequations}
\end{theorem}
\begin{proof}
The equation \eqref{sub1} is slightly different from the equation \eqref{identity}, where the former equation is a finite truncation for the latter one. Moreover, from equations \eqref{pt} and \eqref{dpt}, we have
\begin{align*}
\mathscr{J}_n^{(\alpha)}\Psi_n^{(\alpha)}=x\Psi_n^{(\alpha)}=x\mathscr{L}_n^{(\alpha)}\cdots\mathscr{L}_n^{(\alpha+\theta-1)}\Psi_{n}^{(\alpha+\theta)}=\mathscr{L}_n^{(\alpha)}\cdots\mathscr{L}_n^{(\alpha+\theta-1)}\mathscr{R}_n^{(\alpha)}\Psi_n^{(\alpha)}.
\end{align*}
As the above equation is valid for all $x\in\mathbb{R}$, we know that the equation \eqref{sub2} is true. Since $\mathscr{L}_n^{(\alpha)}$ is invertible, we have
$
\Psi_n^{(\alpha+1)}=\left(
\mathscr{L}_n^{(\alpha)}
\right)^{-1}\Psi_n^{(\alpha)}
$. Then from the equation $x\Psi_n^{(\alpha+1)}=\mathscr{J}_n^{(\alpha+1)}\Psi_n^{(\alpha+1)}$, we obtain that
\begin{align*}
\left(
\mathscr{L}_n^{(\alpha)}
\right)^{-1}\mathscr{J}_n^{(\alpha)}\Psi_n^{(\alpha)}=\mathscr{J}_n^{(\alpha+1)}\left(
\mathscr{L}_n^{(\alpha)}
\right)^{-1}\Psi_n^{(\alpha)}.
\end{align*}
This equation is equivalent to \eqref{sub3} since it is valid for all $x\in\mathbb{R}$.
\end{proof}

In fact, the equation \eqref{sub3} implies that the eigenvalues of $\mathscr{J}_{n}^{(\alpha)}$ are invariant under the evolution from $\alpha$ to $\alpha+1$ by the non-commutative discrete hungry Toda-II lattice \eqref{toda2}, and this discrete lattice equation could be viewed as the recurrence algorithm for the {\emph{generalized block LR transformation}}
\begin{align}\label{glr}
\left\{\begin{array}{ll}
\mathscr{J}_{n}^{(\alpha)}=\mathscr{L}_{n}^{(\alpha)}\mathscr{L}_{n}^{(\alpha+1)}\cdots\mathscr{L}_{n}^{(\alpha+\theta-1)}\mathscr{R}_{n}^{(\alpha)},\\
\mathscr{J}_{n}^{(\alpha+\theta)}=\mathscr{R}_{n}^{(\alpha)}\mathscr{L}_{n}^{(\alpha)}\mathscr{L}_{n}^{(\alpha+1)}\cdots\mathscr{L}_{n}^{(\alpha+\theta-1)}.
\end{array}\right.
\end{align}
Noting that when $\theta=1$, \eqref{glr} reduces to the block LR transformation and the corresponding algorithm \eqref{toda2} becomes block qd-algorithm \cite{wynn1963nc}. Thus we call this recurrence algorithm \eqref{toda2} the \emph{generalized block qd-algorithm}.
\subsection{Convergence of the generalized block qd-algorithm}\label{sec.4.2}
In this subsection, we prove the convergence of the generalized block qd-algorithm under certain assumptions.
Let $\mathscr{J}_{n}^{(0)}$ be an $n\times n$ block Hessenberg matrix given by $\mathscr{J}_{n}^{(0)}=\mathscr{L}_{n}^{(0)}\cdots\mathscr{L}_{n}^{(\theta-1)}\mathscr{R}_{n}^{(0)}$, with
\begin{align}\label{lnrn}
\mathscr{R}_{n}^{(\alpha)}=\left(\begin{array}{ccccc}
q_{1}^{(\alpha)}&\mathbb{I}_{p}&&&\\
&q_{2}^{(\alpha)}&\mathbb{I}_{p}&&\\
&&\ddots&\ddots&\\
&&&q_{n-1}^{(\alpha)}&\mathbb{I}_{p}\\
&&&&q_{n}^{(\alpha)}
\end{array}\right),\quad
\mathscr{L}_{n}^{(\alpha)}=\left(\begin{array}{ccccc}
\mathbb{I}_{p}&&&&\\
e_{1}^{(\alpha)}&\mathbb{I}_{p}&&&\\
&e_{2}^{(\alpha)}&\mathbb{I}_{p}&&\\
&&\ddots&\ddots&\\
&&&e_{n-1}^{(\alpha)}&\mathbb{I}_{p}
\end{array}\right),
\end{align}
whose entries are $p\times p$ blocks. Then by the generalized block qd-algorithm \eqref{toda2}, all entries $\{q_{1}^{(\alpha+1)}, \cdots, q_{n}^{(\alpha+1)}, e_{1}^{(\alpha+\theta)}, \cdots, e_{n-1}^{(\alpha+\theta)}\}_{\alpha\in\mathbb{N}}$ can be obtained. The following theorems describe the asymptotic behaviour of $\mathscr{J}_{n}^{(\theta\alpha)}$ as $\alpha\rightarrow\infty$.
\begin{theorem}\label{th.1}
Suppose that the modulus of eigenvalues $\lambda_{i}$ of $\mathscr{J}_{n}^{(0)}$ are distinct and arranged in descending order $|\lambda_{i}|> |\lambda_{i+1}|, i=1,2\cdots np-1$.
Let $D_{n}$ be the diagonal matrix containing all the eigenvalues such that $D_{n}(i,i)=\lambda_{i},\,i=1, \,2, \,\cdots, \,np$, and $X_{n}$ be the matrix of the eigenvectors associated to $D_{n}$. If the $n-1$ dominant principal block minors of $X_{n}$ and $X_{n}^{-1}$ are nonzero, then
\begin{align}\label{limit_j1}
\lim_{\alpha\rightarrow\infty}\mathscr{J}_{n}^{(\theta\alpha)}=\left(\begin{array}{cccc}
q_{1}^{(\theta\alpha)}&\mathbb{I}_{p}&&\\
&\ddots&\ddots&\\
&&q_{n-1}^{(\theta\alpha)}&\mathbb{I}_{p}\\
&&&q_{n}^{(\theta\alpha)}
\end{array}\right).
\end{align}
\end{theorem}
\begin{proof}
Firstly, according to \eqref{sub3}, we have
\begin{align*}
\mathscr{J}_{n}^{(\alpha)}=&(\mathscr{L}_{n}^{(\alpha-1)})^{-1}(\mathscr{L}_{n}^{(\alpha-2)})^{-1}\cdots(\mathscr{L}_{n}^{(0)})^{-1}
\mathscr{J}_{n}^{(0)}\mathscr{L}_{n}^{(0)}\mathscr{L}_{n}^{(1)}\cdots\mathscr{L}_{n}^{(\alpha-1)}\\
=&(P_{n}^{(\alpha)})^{-1}\mathscr{J}_{n}^{(0)}P_{n}^{(\alpha)}.
\end{align*}
If we set $Q_{n}^{(\alpha)}=\mathscr{R}_{n}^{(\theta(\alpha-1))}\cdots\mathscr{R}_{n}^{(\theta)}\mathscr{R}_{n}^{(0)}$, with the help of \eqref{sub1} and \eqref{sub2} we have $(\mathscr{J}_{n}^{(0)})^{\alpha}=P_{n}^{(\theta\alpha)}Q_{n}^{(\alpha)}$.

Next, we denote $X_{n}^{-1}=Y_{n}$, then $\mathscr{J}_{n}^{(0)}=X_{n}D_{n}Y_{n}$.
Since the $n-1$ dominant principal block minors of $Y_{n}$ are nonzero, it possesses a block $L_{Y}U_{Y}$ decomposition \cite[Theorem 3.7]{quarteroni2000book}, where $L_{Y}$ is a unit block lower triangular matrix. Therefore,
\begin{align}\label{ii}
(\mathscr{J}{n}^{(0)})^{\alpha}=X{n}D_{n}^{\alpha}Y_{n}=X_{n}(D_{n}^{\alpha}L_{Y}D_{n}^{-\alpha})D_{n}^{\alpha}U_{Y}.
\end{align}

If matrix $D_{n}$ is partitioned into blocks such that the diagonal blocks are of size $p \times p$, that is to say, $D_{n}$ can be expressed as follows,
\[
D_{n}=\left(\begin{array}{cccc}
d_{1}&&&\\
&d_{2}&&\\
&&\ddots&\\
&&&d_{n}
\end{array}
\right),\quad d_{k}=\left(\begin{array}{cccc}
\lambda_{(k-1)p+1}&&&\\
&\lambda_{(k-1)p+2}&&\\
&&\ddots&\\
&&&\lambda_{kp}
\end{array}
\right),\quad k=1,2,\cdots,n,
\]
then it is easy to see that $D_{n}^{\alpha}L_{Y}D_{n}^{-\alpha}$ is a unit block lower triangular matrix and can be written as
\begin{align}
D_{n}^{\alpha}L_{Y}D_{n}^{-\alpha}=I+F_{n}^{(\alpha)}, \quad F_{n}^{(\alpha)}=\left\{\begin{array}{ll}
d_{i}^{\alpha}L_{Y}(i,j)d_{j}^{-\alpha},& for \ i>j,\\
0,& for \ i<j,
\end{array}\right.\quad i,\ j=1,\cdots,n.
\end{align}
Thus for $i>j$, the entries of $L_{Y}(i,j)$ in the $a$-th row and $b$-th column are multiplied by $(\frac{\lambda_{(i-1)p+a}}{\lambda_{(j-1)p+b}})^{\alpha}$, $a,\ b=1,\cdots,p.$
Since $|\lambda_{(i-1)p+a}|<|\lambda_{(j-1)p+b}|$, then
\begin{align*}
\lim_{\alpha\rightarrow\infty}(\frac{\lambda_{(i-1)p+a}}{\lambda_{(j-1)p+b}})^{\alpha}=0,
\end{align*}
so that
\begin{align*}
\lim_{\alpha\rightarrow\infty}F_{n}^{(\alpha)}=0.
\end{align*}
Since $X_{n}$ also has a block $L_{X}U_{X}$ decomposition, we obtain
\begin{align}
(\mathscr{J}_{n}^{(0)})^{\alpha}=L_{X}U_{X}(I+F_{n}^{(\alpha)})D_{n}^{\alpha}U_{Y}=L_{X}(I+U_{X}F_{n}^{(\alpha)}U_{X}^{-1})U_{X}D_{n}^{\alpha}U_{Y}.
\end{align}
If we set $G_{n}^{(\alpha)}=U_{X}F_{n}^{(\alpha)}U_{X}^{-1}$,
\begin{align}
(\mathscr{J}_{n}^{(0)})^{\alpha}=L_{X}(I+G_{n}^{(\alpha)})U_{X}D_{n}^{\alpha}U_{Y},\quad \lim_{\alpha\rightarrow\infty}G_{n}^{(\alpha)}=0.
\end{align}
For $k$ sufficiently large the matrix $I+G_{n}^{(\alpha)}$ is strictly diagonally dominant, therefore this matrix has a block LU decomposition, which will be written as
\begin{align}
I+G_{n}^{(\alpha)}=(I+L_{G}^{(\alpha)})(I+U_{G}^{(\alpha)}),\quad \lim_{\alpha\rightarrow\infty}L_{G}^{(\alpha)}=\lim_{\alpha\rightarrow\infty}U_{G}^{(\alpha)}=0.
\end{align}
Hence
\begin{align}
(\mathscr{J}_{n}^{(0)})^{\alpha}=L_{X}(I+L_{G}^{(\alpha)})(I+U_{G}^{(\alpha)})U_{X}D_{n}^{\alpha}U_{Y}=P_{n}^{(\theta\alpha)}Q_{n}^{(\alpha)},
\end{align}
then we get
\[
\lim_{\alpha\rightarrow\infty}P_{n}^{(\theta\alpha)}=L_{X}.
\]
Since
\begin{align}
\mathscr{J}_{n}^{(\theta\alpha)}=(P_{n}^{(\theta\alpha)})^{-1}\mathscr{J}_{n}^{(0)}P_{n}^{(\theta\alpha)}=(P_{n}^{(\theta\alpha)})^{-1}L_{X}U_{X}D_{n}U_{X}^{-1}L_{X}^{-1}P_{n}^{(\theta\alpha)}
,\end{align}
we have
\[
\lim_{\alpha\rightarrow\infty}\mathscr{J}_{n}^{(\theta\alpha)}=U_{X}D_{n}U_{X}^{-1},
\]
which is a block upper triangular matrix. However, $\mathscr{J}_{n}^{(\theta\alpha)}$ is a block lower Hessenberg matrix with the form of \eqref{pt}, hence the result \eqref{limit_j1} holds.
\end{proof}

\begin{theorem}\label{th.2}
If the eigenvalue $\lambda_{i}, i=1,2,\cdots, s (s<np)$ of $\mathscr{J}_{n}^{(0)}$ has a multiplicity $\eta_{i}$ with $\sum_{i=1}^{s}\eta_{i}=np$, and $\lambda_{i}$'s are arranged in descending order of modulus $|\lambda_{i}| > |\lambda_{i+1}|$. In this case, $D_n=diag(\lambda_{1},\cdots,\lambda_{1},\cdots,\lambda_{s},\cdots,\lambda_{s})$, and $X_n$ is the corresponding eigenvector matrix .
If the $n-1$ dominant principal block minors of $X_{n}$, $X_{n}^{-1}$ and $X_{n}\tilde{L}$ are nonzero ($\tilde{L}$ is obtained as explained in the proof from the relation \eqref{iii}), then
\begin{align}\label{limit_j2}
\lim_{\alpha\rightarrow\infty}\mathscr{J}_{n}^{(\theta\alpha)}=\left(\begin{array}{cccc}
q_{1}^{(\theta\alpha)}&\mathbb{I}_{p}&&\\
&\ddots&\ddots&\\
&&q_{n-1}^{(\theta\alpha)}&\mathbb{I}_{p}\\
&&&q_{n}^{(\theta\alpha)}
\end{array}\right).
\end{align}
\end{theorem}
\begin{proof}
Firstly, according to \eqref{sub3}, we have
\begin{align*}
\mathscr{J}_{n}^{(\alpha)}=&(\mathscr{L}_{n}^{(\alpha-1)})^{-1}(\mathscr{L}_{n}^{(\alpha-2)})^{-1}\cdots(\mathscr{L}_{n}^{(0)})^{-1}
\mathscr{J}_{n}^{(0)}\mathscr{L}_{n}^{(0)}\mathscr{L}_{n}^{(1)}\cdots\mathscr{L}_{n}^{(\alpha-1)}\\
=&(P_{n}^{(\alpha)})^{-1}\mathscr{J}_{n}^{(0)}P_{n}^{(\alpha)}.
\end{align*}
If we set $Q_{n}^{(\alpha)}=\mathscr{R}_{n}^{(\theta(\alpha-1))}\cdots\mathscr{R}_{n}^{(\theta)}\mathscr{R}_{n}^{(0)}$, with the help of \eqref{sub1} and \eqref{sub2} we have $(\mathscr{J}_{n}^{(0)})^{\alpha}=P_{n}^{(\theta\alpha)}Q_{n}^{(\alpha)}$.

Next, we denote $X_{n}^{-1}=Y_{n}$, then $\mathscr{J}_{n}^{(0)}=X_{n}D_{n}Y_{n}$.
Since $Y_{n}$ has a block $L_{Y}U_{Y}$ decomposition, where $L_{Y}$ is a unit block lower triangular matrix, then
\begin{align}\label{XDY}
(\mathscr{J}_{n}^{(0)})^{\alpha}=X_{n}D_{n}^{\alpha}Y_{n}=X_{n}(D_{n}^{\alpha}L_{Y}D_{n}^{-\alpha})D_{n}^{\alpha}U_{Y}.
\end{align}

The above process is exactly the same as that of Theorem \ref{th.1}.

Next, for $v\in\mathbb{N}$, $1\leq v\leq s$, we assume that $\lambda_{v}$ is an eigenvalue of $\mathscr{J}_{n}^{(0)}$ with multiplicity $\eta_{v}>1$, then for $\forall m,l\in\{\sigma_{v-1}+1,\cdots,\sigma_{v}\}$, we have $\frac{\lambda_{m}}{\lambda_{l}}=1$.
Let's use a special decomposition of $D_{n}^{\alpha}L_{Y}D_{n}^{-\alpha}$ in \eqref{XDY},
\begin{align}\label{iii}
D_{n}^{\alpha}L_{Y}D_{n}^{-\alpha}=\tilde{L}+\tilde{F}_{n}^{(\alpha)},
\end{align}
where $\tilde{L}$ is a unit block lower triangular matrix. $\tilde{L}$ has block entries only containing the scalar entries of $L_{Y}$ which are multiplied by all the $(\frac{\lambda_{a}}{\lambda_{b}})^{\alpha}$ such that $\frac{\lambda_{a}}{\lambda_{b}}=1$, and $\tilde{F}_{n}^{(\alpha)}$ only has block entries whose scalar entries are multiplied by all the $(\frac{\lambda_{a}}{\lambda_{b}})^{\alpha}$ such that $\lim_{\alpha\rightarrow\infty}(\frac{\lambda_{a}}{\lambda_{b}})^{\alpha}=0$, $1\leq a,b\leq s$, hence
\[
\lim_{\alpha\rightarrow\infty}\tilde{F}_{n}^{(\alpha)}=0.
\]

Then \eqref{XDY} becomes
\[
(\mathscr{J}_{n}^{(0)})^{\alpha}=X_{n}\tilde{L}(I+\tilde{L}^{-1}\tilde{F}_{n}^{(\alpha)})D_{n}^{\alpha}U_{Y}, \quad \lim_{\alpha\rightarrow\infty}\tilde{L}^{-1}\tilde{F}_{n}^{(\alpha)}=0.
\]
Since $X_{n}\tilde{L}$ has a $\tilde{L}_{X}\tilde{U}_{X}$ decomposition, then
\[
(\mathscr{J}_{n}^{(0)})^{\alpha}=\tilde{L}_{X}(I+\tilde{U}_{X}\tilde{L}^{-1}\tilde{F}_{n}^{(\alpha)}\tilde{U}_{X}^{-1})
\tilde{U}_{X}D_{n}^{\alpha}U_{Y}.
\]
For $\alpha$ sufficiently large the matrix $I+\tilde{U}_{X}\tilde{L}^{-1}\tilde{F}_{n}^{(\alpha)}\tilde{U}_{X}^{-1}$ is strictly diagonally dominant, thus for $\alpha$ sufficiently large,
\begin{align*}
&I+\tilde{U}_{X}\tilde{L}^{-1}\tilde{F}_{n}^{(\alpha)}\tilde{U}_{X}^{-1}=(I+\tilde{L}_{G}^{(\alpha)})(I+\tilde{U}_{G}^{(\alpha)}),\\ &\lim_{\alpha\rightarrow\infty}\tilde{L}_{G}^{(\alpha)}=\lim_{\alpha\rightarrow\infty}\tilde{U}_{G}^{(\alpha)}=0.
\end{align*}
Then we can still conclude that
\[
\lim_{\alpha\rightarrow\infty}P_{n}^{(\theta\alpha)}=\tilde{L}_{X}.
\]
Thus
\begin{align}
\begin{aligned}
\lim_{\alpha\rightarrow\infty}\mathscr{J}_{n}^{(\theta\alpha)}=&(P_{n}^{(\theta\alpha)})^{-1}\mathscr{J}_{n}^{(0)}P_{n}^{(\theta\alpha)}\\
=&\tilde{U}_{X}(\tilde{L}^{-1}D_{n}\tilde{L})\tilde{U}_{X}^{-1}.
\end{aligned}
\end{align}
Then we cut the matrix $\tilde{L}^{-1}D_{n}\tilde{L}$ in blocks, where the diagonal blocks are $\eta_{i}\times\eta_{i},\,i=1,2,\cdots,s$, we can see that $\tilde{L}^{-1}D_{n}\tilde{L}=D_{n}$.

Therefore, the block lower Hessenberg matrix $\mathscr{J}_{n}^{(\theta\alpha)}$ tends to a block upper triangular matrix as $\alpha\rightarrow\infty$, that is to say \eqref{limit_j2} holds.
\end{proof}
These two theorems imply that under certain assumptions, the generalized qd-algorithm could transform the block Hessenberg matrix $\mathscr{J}_{n}^{(0)}$ into a block triangle matrix while preserving all the eigenvalues, and thus can be viewed as a pre-processing algorithm for matrix eigenvalues. The specific process of this algorithm is as follows.
\begin{algorithm}[!h]\label{algorithm1}
    \SetAlgoLined 
	\caption{the generalized block qd-algorithm}
	\KwIn{$\theta,n,p$, $q_{m} (m=1,2,\cdots,n)$ , $e_{m}^{(i)}(m=1,2,\cdots,n-1;i=0,1,\cdots,\theta-1)$ and maximum iteration number $N$ }
	\KwOut{$\lambda_{m,l}(m=1,2,\cdots,n;l=1,2,\cdots,p)$}
\For{$i=0:\theta-1$}
{$e_{0}^{(i)}=e_{n}^{(i)}=0;$}
\For{$k=1:N$}
{   \For{i=$0:\theta-1$}
    {
       {$q_{1}=q_{1}+e_{1}^{(i)}-e_{0}^{(i)};$}\\
		\For{$m=2:n$}
        {
		$e_{m-1}^{(i)}=q_{m}\cdot e_{m-1}^{(i)}\cdot q_{m-1}^{-1},$\\
        $q_{m}=q_{m}+e_{m}^{(i)}-e_{m-1}^{(i)};$
        }
    }
}\For
       {$m=1:n$}{computing the eigenvalues of $ q_{m}:\lambda_{m,l},l=1,\cdots,p$;}
\end{algorithm}
\begin{remark}
  Of course, we could also design a pre-processing algorithm starting from the discrete non-commutative hungry Toda-I lattice. However, since this approach does not fundamentally differ, so we have not listed it here.
\end{remark}

%


\section{Numerical experiments}\label{sec.5}
In this section, we will present some numerical examples to demonstrate the performance of the generalized block qd-algorithm. All computations are carried out in Matlab R2019b (academic use) on a computer with a 2.20 GHz CPU and 4 GB main memory.
\begin{example}\label{eg1}
Let us consider the first case where
\begin{align*}
q_{1}^{(0)}=\left(\begin{array}{ccc}
7 & 0 & 1\\
2 & 9 & 8\\
-3& 1 & 3
\end{array}\right),\quad
q_{2}^{(0)}=\left(\begin{array}{ccc}
8 & 5 & 1\\
1 & 8 & 2\\
2 & 0 & 10
\end{array}\right),\quad
q_{3}^{(0)}=\left(\begin{array}{ccc}
5 & 1 & 1\\
2 & 7 &-6\\
0 & 3 & 8
\end{array}\right),
\end{align*}
\begin{align*}
q_{4}^{(0)}=\left(\begin{array}{ccc}
6 & 0 & 2\\
0 & 8 & 1\\
3 & 3 & 5
\end{array}\right),\quad
q_{5}^{(0)}=\left(\begin{array}{ccc}
5 & 0 & 1\\
1 & 5 & 6\\
2 & 1 & -10
\end{array}\right),\quad
e_{1}^{(0)}=\left(\begin{array}{ccc}
8 & 1 & 1\\
10& 8 & 2\\
2 & 0 & 9
\end{array}\right),
\end{align*}
\begin{align*}
e_{2}^{(0)}=\left(\begin{array}{ccc}
9 & 5 & 1\\
1 & 6 & 0\\
1 & 1 & 3
\end{array}\right),\quad
e_{3}^{(0)}=\left(\begin{array}{ccc}
6 & 0 &-2\\
-2& 4 & 1\\
6 & 1 &10
\end{array}\right),\quad
e_{4}^{(0)}=\left(\begin{array}{ccc}
15& 0 & 2\\
0 & 4 & 1\\
10& 2 & 8
\end{array}\right),
\end{align*}
\begin{align*}
e_{1}^{(1)}=\left(\begin{array}{ccc}
2 &-1 & 3\\
2 & 2 & 2\\
1 & 4 & 3
\end{array}\right),\quad
e_{2}^{(1)}=\left(\begin{array}{ccc}
1 & 0 & 3\\
2 & 6 & 3\\
3 &-5 &10
\end{array}\right),
\end{align*}
\begin{align*}
e_{3}^{(1)}=\left(\begin{array}{ccc}
6 & 0 & 1\\
0 & 8 & 0\\
3 & 0 & 10
\end{array}\right),\quad
e_{4}^{(1)}=\left(\begin{array}{ccc}
4 & 0 & 6\\
1 & 4 & 1\\
-3& 0 & 2
\end{array}\right)
\end{align*}
in the discrete non-commutative hungry Toda-II lattice \eqref{toda2} with $\theta=2$, $n=5$ and $p=3$. Then, it follows that
\begin{align}\label{15}
\mathscr{J}_{n}^{(0)}&=\mathscr{L}_{n}^{(0)}\mathscr{L}_{n}^{(1)}\mathscr{R}_{n}^{(0)}\\
&=\left(\begin{array}{ccccccccccccccc}
7 & 0 & 1 & 1 & 0 & 0 & 0 & 0 & 0 & 0 & 0 & 0 & 0 & 0 & 0\\
2 & 9 & 8 & 0 & 1 & 0 & 0 & 0 & 0 & 0 & 0 & 0 & 0 & 0 & 0\\
-3& 1 & 3 & 0 & 0 & 1 & 0 & 0 & 0 & 0 & 0 & 0 & 0 & 0 & 0\\
58& 4 &22 &18 & 5 & 5 & 1 & 0 & 0 & 0 & 0 & 0 & 0 & 0 & 0\\
92& 94&104&13 &18 & 6 & 0 & 1 & 0 & 0 & 0 & 0 & 0 & 0 & 0\\
-7& 48&71 & 5 & 4 &22 & 0 & 0 & 1 & 0 & 0 & 0 & 0 & 0 & 0\\
93& 85&189&122& 95&100& 15& 6 & 5 & 1 & 0 & 0 & 0 & 0 & 0\\
75&114&147&56 &122&72 & 5 & 19& -3& 0 & 1 & 0 & 0 & 0 & 0\\
33&31 &153&61 & 1 &140& 4 & -1& 21& 0 & 0 & 1 & 0 & 0 & 0\\
0 & 0 & 0 & 6 & 80& 0 & 60& 19& 2 & 18& 0 & 1 & 1 & 0 & 0\\
0 & 0 & 0 &123&197&207& 23&104&-50& -2& 20& 2 & 0 & 1 & 0\\
0 & 0 & 0 &502&-162&1160&85&32&284& 12& 4 & 25& 0 & 0 & 1\\
0 & 0 & 0 & 0 & 0 & 0 &480&201&376&234& 24&113& 24& 0 & 9\\
0 & 0 & 0 & 0 & 0 & 0 & 79&257&-109&15&102&30 & 2 &13 & 8\\
0 & 0 & 0 & 0 & 0 & 0 &452&466&708&156& 62&156& 9 & 3 & 0\\
\end{array}\right).
\end{align}
\end{example}
Numerical results are shown below.
The asymptotic behaviour of $\|q_{k}\|_{F}$ and $\|e_{k}\|_{F}$ are shown in figure \ref{1111}. It is obvious from figure \ref{1111} that $\|q_{k}\|_{F}$ converges to some constant as iteration number $N$ increases and $\|e_{k}\|_{F}$ converges to zero as iteration number $N$ increases. This indicates that $\mathscr{J}_{n}^{(N)}$ tends to a block upper triangular matrix.

\begin{figure*}[htb]
\centering
\begin{subfigure}[b]{0.49\textwidth}
\includegraphics[width=\textwidth]{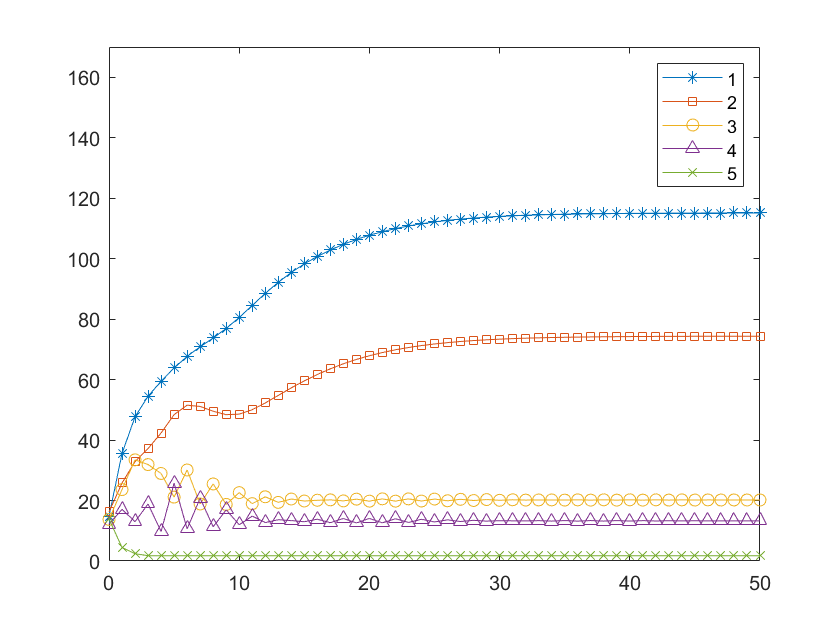}
\caption{
}
\end{subfigure}
\begin{subfigure}[b]{0.49\textwidth}
\includegraphics[width=\textwidth]{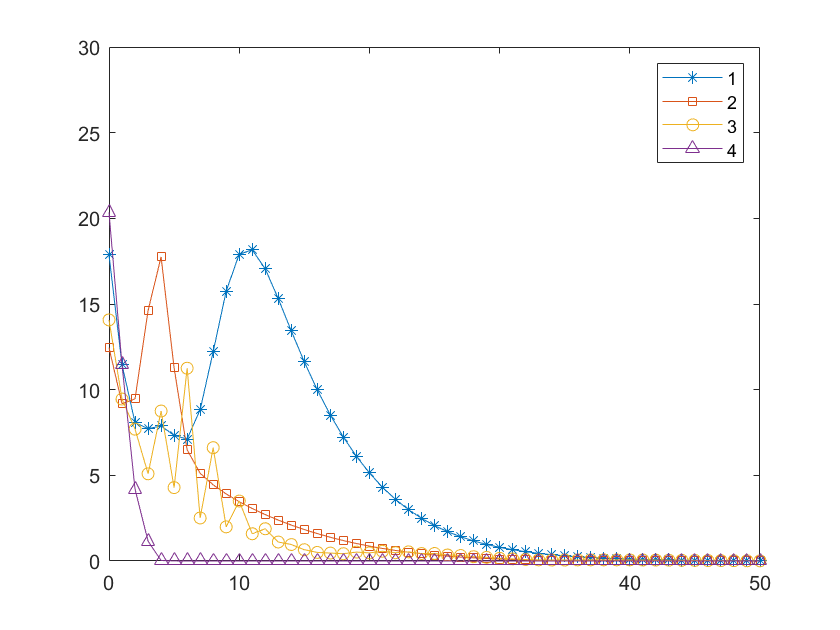}
\caption{
}
\end{subfigure}
\caption{Asymptotic behavior in Example \ref{eg1}. (A): the iteration number $N$ (x-axis) and the value of $\|q_{k}\|_{F}$ for $k = 1,2,3,4,5$ (y-axis). (B): the iteration number $N$ (x-axis) and the value of $\|e_{k}\|_{F}$ for $k = 1,2,3,4$ (y-axis).}
\label{1111}
\end{figure*}
Next, we investigate the accuracy of eigenvalues computed under the generalized qd-algorithm's pre-processing.
The first column of table \ref{qqq} denotes the eigenvalues of $\mathscr{J}_{n}^{(0)}$ in (\ref{15}) obtained by Matlab function Eigenvalues, while the second column describes the results computed with preconditioning.
It is observed that the generalized qd-algorithm can indeed help us compute eigenvalues with high accuracy.

\begin{table}[h]
	\caption{Computed eigenvalues of $\mathscr{J}_{n}^{(0)}$ in Example \ref{eg1}.}
    \centering
    \scalebox{0.8}{
	\begin{tabular}{ll}
		\toprule  
		Matlab & generalized qd-algorithm \\ \hline
		0.076645990783392&0.076645990783389\\
        0.353812648575613&0.353812648575593\\
        0.742197652318054 &0.742197652318063 \\
        3.798992251521434 $-$ 1.205082251424433i&3.798992251521439 $-$ 1.205082251424437i \\
        3.798992251521434 $+$ 1.205082251424433i &3.798992251521439 $+$ 1.205082251424437i\\
        $-$7.000446257430013 &$-$7.000446257620274\\
        7.765597517750727 $-$ 1.441904857381309i &7.765597517614697 $-$ 1.441904857525976i \\
        7.765597517750727 $+$ 1.441904857381309i & 7.765597517614697 $+$ 1.441904857525976i \\
        15.132368187137137 &15.132368187137070\\
        17.792893502360450 $-$ 1.598641807492877i&17.792893502360485 $-$ 1.598641807492811i\\
        17.792893502360450 $+$ 1.598641807492877i&17.792893502360485 $+$ 1.598641807492811i\\
        30.579683715887128&30.579683715886933\\
        37.067180191075906&37.067180191075998\\
        41.776330813552491 &41.776330813552505 \\
        54.557260514835193&54.557260514835086\\
		\bottomrule  
	\end{tabular}
   }
   \label{qqq}
\end{table}

\begin{example}\label{eg2}
 Let us consider the second case where
 \begin{align*}
e_{1}^{(0)}=\left(\begin{array}{cc}
6 & 8 \\
2 & 8
\end{array}\right),\quad
e_{2}^{(0)}=\left(\begin{array}{cc}
8 & -3 \\
1 & 12
\end{array}\right),\quad
e_{3}^{(0)}=\left(\begin{array}{cc}
1 & 3 \\
-5 & 1
\end{array}\right),\quad
e_{1}^{(1)}=\left(\begin{array}{cc}
1 & 2 \\
-2 & 8
\end{array}\right),
\end{align*}
\begin{align*}
e_{2}^{(1)}=\left(\begin{array}{cc}
8 & 2 \\
3 & 12
\end{array}\right),\quad
e_{3}^{(1)}=\left(\begin{array}{cc}
10 & -4 \\
3 & 12
\end{array}\right),
e_{1}^{(2)}=\left(\begin{array}{cc}
1 & 2 \\
0 & 1
\end{array}\right),
\end{align*}
\begin{align*}
e_{2}^{(2)}=\left(\begin{array}{cc}
3 & 4 \\
1 & 6
\end{array}\right),\quad
e_{3}^{(2)}=\left(\begin{array}{cc}
1 & 2 \\
1 & 3
\end{array}\right),\quad
q_{1}^{(0)}=\left(\begin{array}{cc}
4 & 4 \\
1 & 2
\end{array}\right),
\end{align*}
\begin{align*}
q_{2}^{(0)}=\left(\begin{array}{cc}
5 & 6 \\
1 & 6
\end{array}\right),\quad
q_{3}^{(0)}=\left(\begin{array}{cc}
2 & 1 \\
2 & 10
\end{array}\right),\quad
q_{4}^{(0)}=\left(\begin{array}{cc}
7 & -2 \\
1 & 8
\end{array}\right)
\end{align*}
in the discrete non-commutative hungry Toda-II lattice \eqref{toda2} with $\theta=3$, $n=4$ and $p=2$. Then, it follows that
\begin{align}\label{8}
\mathscr{J}_{n}^{(0)}&=\mathscr{L}_{n}^{(0)}\mathscr{L}_{n}^{(1)}\mathscr{L}_{n}^{(2)}\mathscr{R}_{n}^{(0)}\\
&=\left(\begin{array}{cccccccc}
4  &   4   &  1  &   0  &   0  &   0  &   0  &   0\\
1   &  2   &  0  &   1  &   0  &   0  &   0  &   0\\
44  &  56  & 13  &  18  &   1  &   0  &   0  &   0\\
17 &   34  & 1   & 23   &  0   &  1  &   0   &  0\\
143  & 166 &  128&   155&    21&     4  &   1  &   0\\
54  & 184  &  36 &  340 &    7 &   40  &   0  &   1\\
140 &  212 &  338&  822 &   75 &   98  &  19  &  -1\\
-220 & -292& -115&   180&     0&   231  &   0 &   24\\
\end{array}\right).
\end{align}
\end{example}

Numerical results are shown below.
The asymptotic behaviour of $\|q_{k}\|_{F}$ and $\|e_{k}\|_{F}$ are shown in figure \ref{6666}. It is obvious from figure \ref{6666} that $\|q_{k}\|_{F}$ converges to some constant as iteration number $N$ increases, $\|e_{k}\|_{F}$ converges to zero as iteration number $N$ increases.
\begin{figure*}[htb]
\centering
\begin{subfigure}[b]{0.49\textwidth}
\includegraphics[width=\textwidth]{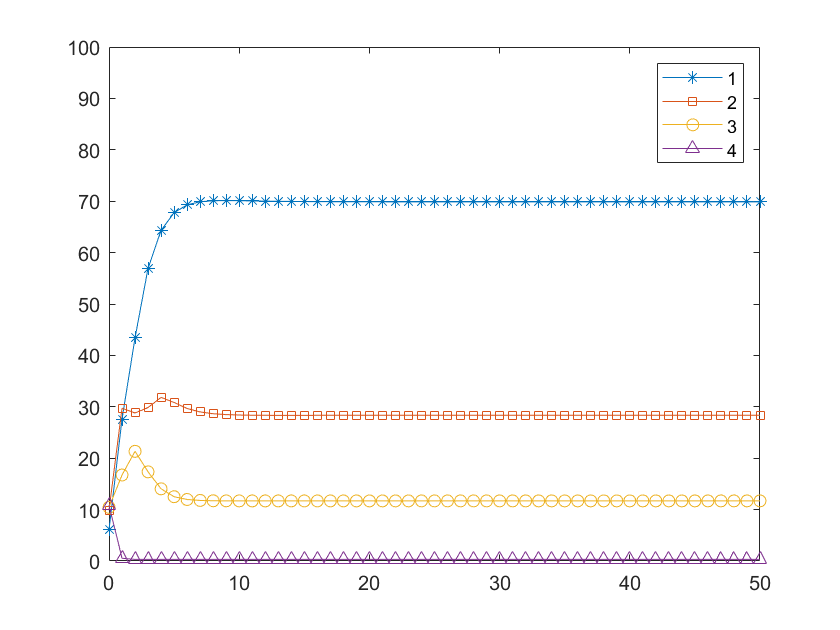}
\caption{
}
\end{subfigure}
\begin{subfigure}[b]{0.49\textwidth}
\includegraphics[width=\textwidth]{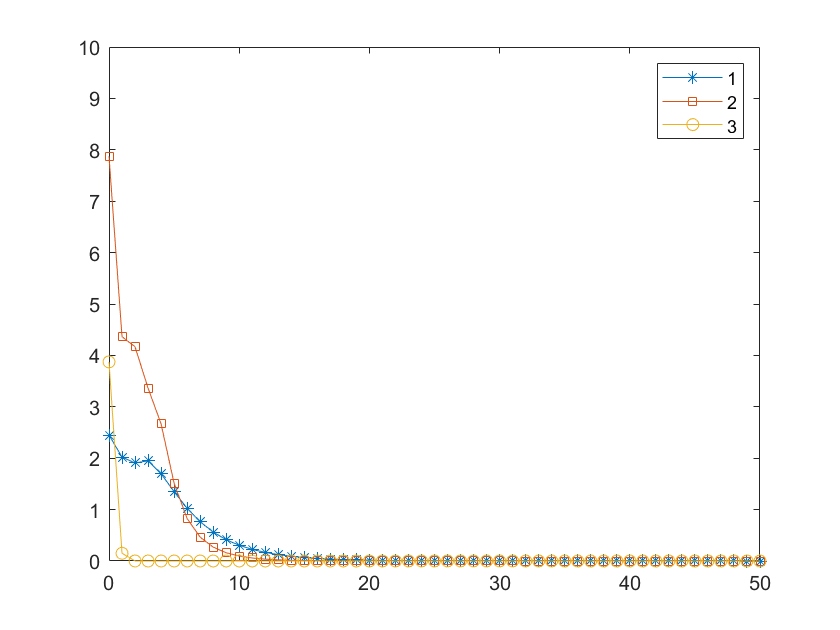}
\caption{
}
\end{subfigure}
\caption{Asymptotic behaviour in Example \ref{eg1}. (A): the iteration number $N$ (x-axis) and the value of $\|q_{k}\|_{F}$ for $k = 1,2,3,4$ (y-axis). (B): the iteration number $N$ (x-axis) and the value of $\|e_{k}\|_{F}$ for $k = 1,2,3$ (y-axis).}
\label{6666}
\end{figure*}

Table \ref{8888} displays the computed eigenvalues by Matlab function Eigenvalues and generalized block qd-algorithm, respectively. It is observed that the generalized block qd-algorithm can accurately compute both the real and imaginary parts of eigenvalues.

\begin{table}[h]
	\caption{Computed eigenvalues of $\mathscr{J}_{n}^{(0)}$ in Example \ref{eg2}.}
    \centering
    \scalebox{0.8}{
	\begin{tabular}{ll}
		\toprule  
		Matlab & generalized qd-algorithm \\ \hline
		0.038550772971028 $-$ 0.033742102109088i&0.038550772971027 $-$ 0.033742102109069i\\
        0.038550772971028 $+$ 0.033742102109088i&0.038550772971027 $+$ 0.033742102109069i\\
        5.782189707201425 &5.782189707201426 \\
        9.513648907489330 &9.513648907489348 \\
        15.452211918682499 &15.452211918682455\\
        23.654136978615185 &23.654136978615195\\
        31.800649952367969 &31.800649952367955 \\
        59.720060989701565 & 59.720060989701587 \\
		\bottomrule  
	\end{tabular}
   }
   \label{8888}
\end{table}
\section{Conclusions and discussions}\label{sec.6}

It is known that matrix-valued orthogonal polynomials and quasi-determinants play crucial roles in non-commutative integrable systems.
In this article, we first introduce the discrete non-commutative hungry Toda lattices and construct their quasi-determinant solutions by considering the compatibility conditions among adjacent families of matrix-valued $\theta$-deformed bi-orthogonal polynomials.
Motivated by the rich applications of integrability and orthogonality in numerical computations, we demonstrate that the corresponding finite-dimensional equation can serve as a pre-processing algorithm for computing eigenvalues of block Hessenberg matrices, called the generalized block qd-algorithm.
We establish the asymptotic convergence of the generalized block qd-algorithm under certain assumptions, and our numerical findings align with the theoretical results.

It should be noted that the block qd-algorithm can be traced back to the paper of Wynn \cite{wynn1963nc} in 1963. Draux and Sadik studied the asymptotic behavior of some eigenvalues of a block tridiagonal positive definite symmetric matrix and proved the convergence of the block qd-algorithm in \cite{draux2010}. Subsequently, they generalized the block qd-algorithm into the case of positive defnite symmetric block band matrix \cite{draux2012}. Although both our algorithm and theirs can be viewed as extensions of the block qd algorithm, the corresponding block matrices have different structures. Furthermore, the generalized block qd-algorithm presented here is derived from the standpoint of non-commutative integrable systems and matrix-valued orthogonal polynomials.

Finally, we would like to remark that the discrete hungry Lotka-Volterra system describing predator-prey interactions is closely related to the discrete hungry Toda equation and can be used to compute banded matrices.
In \cite{gilson23}, the semi-discrete non-commutative hungry Lotka-Volterra lattices have been derived by using a moment modification method. Therefore, it is a natural idea to investigate whether the full-discrte non-commutative hungry Lotka-Volterra lattices have relationships to matrix computation. We will study these problems in the future.
\section*{Acknowledgments}
J.-Q Sun is supported by National Natural Science Foundation of China (Grant No. 12071447). S.-H Li is supported by National Natural Science Foundation of China (Grant Nos. 12175155, 12101432). K.-Y Lu is supported by National Natural Science Foundation of China (Grant No. 12001048).

\appendix
\section{Quasi-determinants}
There have been numerous references regarding with the basic definitions and properties of quasi-determinants, please see \cite{gelfand05,gelfand91,krob95} and references therein.
In general, an $n\times n$ matrix $A$ over a skew field $\mathscr{R}$ has $n^{2}$ quasi-determinants written as $|A|_{i,j}$ for $i,j=1,\ldots,n$, which are also elements of $\mathscr{R}$. They are defined by
\[
|A|_{i,j}=\left|\begin{array}{ccc}
A^{i,j} & c_{j}^{i}\\
r_{i}^{j} & \fbox{$a_{i,j}$}
\end{array}\right|=a_{i,j}-r_{i}^{j}(A^{i,j})^{-1}c_{j}^{i},\quad A^{-1}=(|A|_{j,i}^{-1})_{i,j=1,\ldots,n},
\]
where $r_{i}^{j}$ represents the $ith$ row of $A$ with the $jth$ element removed, $c_{j}^{i}$ represents the $jth$ column with the $ith$ element removed and $A^{i,j}$ represents the submatrix obtained by removing the $ith$ and the $jth$ column from $A$. Assume that $A^{i,j}$ is invertible.
This is in fact the Schur complement of $a_{i,j}$. However, as an analogy of determinant, quasi-determinants has some similar properties with determinants. 

The first one is that a quasi-determinant could be used to describe the solution of a linear system with non-commutative coefficients. This result is used to show the quasi-determinant expressions for matrix-valued orthogonal polynomials.
\begin{proposition}\label{e}
Let $A=(a_{i,j})$ be an $n\times n$ matrix over $\mathscr{R}$. Assume that all quasi-determinants $|A|_{i,j}$ are well defined and invertible. Then
\[
\left\{\begin{array}{ccc}
a_{1,1}x_{1}+\cdots a_{1,n}x_{n}=\xi_{1}\\
\vdots\\
a_{n,1}x_{1}+\cdots a_{n,n}x_{n}=\xi_{n}
\end{array}\right.
\]
has a solution $\xi_{i}\in \mathscr{R}$ if and only if
\[
x_{i}=\sum_{j=1}^{n}|A|_{j,i}^{-1}\xi_{j},\quad i=1,\cdots,n.
\]
\end{proposition}

The second important property for quasi-determinant is the so-called non-commutative Jacobi identity, or non-commutative Sylvester's theorem. The simplest version of this identity is given by
\begin{align}\label{51}
\left|\begin{array}{ccc}
A&B&C\\
D&f&g\\
E&h&\fbox{$i$}
\end{array}\right|=\left|\begin{array}{ccc}
A&C\\
E&\fbox{$i$}
\end{array}\right|-\left|\begin{array}{ccc}
A&B\\
E&\fbox{$h$}
\end{array}\right|\left|\begin{array}{ccc}
A&B\\
D&\fbox{$f$}
\end{array}\right|^{-1}\left|\begin{array}{ccc}
A&C\\
D&\fbox{$g$}
\end{array}\right|.
\end{align}
As a direct corollary, we have the following homological relations
\begin{align}\label{qq}
\begin{aligned}
\left|\begin{array}{ccc}
A&B&C\\
D&f&g\\
E&\fbox{$h$}&i
\end{array}\right|=\left|\begin{array}{ccc}
A&B&C\\
D&f&g\\
E&h&\fbox{$i$}
\end{array}\right|\left|\begin{array}{ccc}
A&B&C\\
D&f&g\\
0&\fbox{$0$}&1
\end{array}\right|,\\
\left|\begin{array}{ccc}
A&B&C\\
D&f&\fbox{$g$}\\
E&h&i
\end{array}\right|=\left|\begin{array}{ccc}
A&B&0\\
D&f&\fbox{$0$}\\
E&h&1
\end{array}\right|\left|\begin{array}{ccc}
A&B&C\\
D&f&g\\
E&h&\fbox{$i$}
\end{array}\right|,
\end{aligned}
\end{align}
which were used to show the quasi-determinant expressions for coefficients in discrete spectral transformations.

\end{document}